%% file: SECKv3.tex
\documentclass[a4paper, 10pt]{article}
\input format.tex
\begin{document}%
\title{Strong chromatic index of $k$-degenerate graphs}
\author{Tao Wang\,\textsuperscript{a, b, }\footnote{{\tt Corresponding
author: wangtao@henu.edu.cn} }\\
{\small \textsuperscript{a}Institute of Applied Mathematics}\\
{\small Henan University, Kaifeng, 475004, P. R. China}\\
{\small \textsuperscript{b}College of Mathematics and Information Science}\\
{\small Henan University, Kaifeng, 475004, P. R. China}}
\date{}
\maketitle
\begin{abstract}%
A {\em strong edge coloring} of a graph $G$ is a proper edge coloring in which every color class is an induced matching. The {\em strong chromatic index} $\chiup_{s}'(G)$ of a graph $G$ is the minimum number of colors in a strong edge coloring of $G$. In this note, we improve a result by D{\k e}bski \etal [Strong chromatic index of sparse graphs, arXiv:1301.1992v1] and show that the strong chromatic index of a $k$-degenerate graph $G$ is at most $(4k-2) \cdot \Delta(G) - 2k^{2} + 1$. As a direct consequence, the strong chromatic index of a $2$-degenerate graph $G$ is at most $6\Delta(G) - 7$, which improves the upper bound $10\Delta(G) - 10$ by Chang and Narayanan [Strong chromatic index of 2-degenerate graphs, J. Graph Theory 73 (2013) (2) 119--126]. For a special subclass of $2$-degenerate graphs, we obtain a better upper bound, namely if $G$ is a graph such that all of its $3^{+}$-vertices induce a forest, then $\chiup_{s}'(G) \leq 4 \Delta(G) -3$; as a corollary, every minimally $2$-connected graph $G$ has strong chromatic index at most $4 \Delta(G) - 3$. Moreover, all the results in this note are best possible in some sense. 
\end{abstract}
\section{Introduction}
A {\em strong edge coloring} of a graph $G$ is a proper edge coloring in which every color class is an induced matching. That is, an edge coloring is {\em strong} if for each edge $uv$, the color of $uv$ is distinct from the colors of the edges (other than $uv$) incident with $N_{G}(u) \cup N_{G}(v)$. The {\em strong chromatic index} $\chiup_{s}'(G)$ of a graph $G$ is the minimum number of colors in a strong edge coloring of $G$. The {\em degree} of a vertex $v$ in $G$, denoted by $\deg(v)$, is the number of incident edges of $v$ in $G$. A vertex of degree $k$, at most $k$ and at least $k$ are called a $k$-vertex, $k^{-}$-vertex and $k^{+}$-vertex, respectively. The graph $(\emptyset, \emptyset)$ is an {\em empty graph}, and $(V, \emptyset)$ is an {\em edgeless graph}. We denote the minimum and maximum degrees of vertices in $G$ by $\delta(G)$ and $\Delta(G)$, respectively. 

In 1985, Edr{\" o}s and Ne{\v s}et{\v r}il \cite{MR975526} constructed graphs with strong chromatic index $\frac{5}{4}\Delta^{2}$ when $\Delta$ is even, $\frac{1}{4}(5\Delta^{2} - 2 \Delta + 1)$ when $\Delta$ is odd. Inspired by their construction, they proposed the following strong edge coloring conjecture.
\begin{conjecture}%
If $G$ is a graph with maximum degree $\Delta$, then
\[
\chiup_{s}'(G) \leq \left \{
\begin{array}{ll}
\frac{5}{4} \Delta^{2}, & \text{if $\Delta$ is even;}\\
\frac{1}{4}(5\Delta^{2} - 2 \Delta + 1), & \text{if $\Delta$ is odd.} \qed
\end{array}
\right. 
\]
\end{conjecture}
A graph is {\em $k$-degenerate} if every subgraph has a vertex of degree at most $k$. Chang and Narayanan \cite{MR3056878} showed the strong chromatic index of a $2$-degenerate graph $G$ is at most $10\Delta(G) - 10$. Recently, Luo and Yu \cite{2012arXiv1212.6092L} improved the upper bound to $8\Delta(G) - 4$. For general $k$-degenerate graphs, D{\k e}bski \etal \cite{2013arXiv1301.1992D} presented an upper bound $(4k - 1) \cdot \Delta(G) - k(2k + 1) + 1$. Very recently, Yu \cite{2012arXiv1212.6093Y} obtained an improved upper bound $(4k-2) \cdot \Delta(G) - 2k^{2} + k + 1$. In this note, we use the method developed in \cite{2013arXiv1301.1992D} and improve the upper bound to $(4k-2) \cdot \Delta(G) - 2k^{2} + 1$. In particular, when $G$ is a $2$-degenerate graph, the strong chromatic index is at most $6\Delta(G) - 7$, which improves the upper bound $10\Delta(G) - 10$ by Chang and Narayanan \cite{MR3056878}. In addition, we show that if $G$ is a graph such that all of its $3^{+}$-vertices induce a forest, then $\chiup_{s}'(G) \leq 4 \Delta(G) -3$.
\section{Results}
\begin{lemma}[Chang and Narayanan \cite{MR3056878}]\label{nice-vertex}%
If $G$ is a $k$-degenerate graph with at least one edge, then there exists a vertex $w$ such that at least $\max\{1, \deg(w) - k\}$ of its neighbors are $k^{-}$-vertices.
\end{lemma}

\begin{theorem}\label{k-degenerate}%
If $G$ is a $k$-degenerate graph with maximum degree $\Delta$ and $k \leq \Delta$, then $\chiup_{s}'(G) \leq (4k-2) \Delta - 2k^{2} + 1$.
\end{theorem}
\begin{proof}%
By the definition of $k$-degenerate graph, every subgraph of $G$ is also a $k$-degenerate graph. We want to obtain a sequence $\Lambda_{1}, \dots, \Lambda_{m}$ of subsets of edges as follows. Let $\Lambda_{0} = \emptyset$. Suppose that $\Lambda_{i-1}$ is well-defined, let $G_{i}$ be the graph $G - (\Lambda_{0} \cup \dots \cup \Lambda_{i-1})$. Notice that $G_{i}$ is an edgeless graph or a $k$-degenerate graph with at least one edge. Denote the degree of $v$ in $G_{i}$ by $\deg_{i}(v)$. If $G_{i}$ has at least one edge, then we choose a vertex $w_{i}$ of $G_{i}$ as described in \autoref{nice-vertex}, and let
\[
\Lambda_{i} = \{\,w_{i}v \mid w_{i}v \in E(G_{i}) \mbox{ and } \deg_{i}(v) \leq k\,\}.
\]
\autoref{nice-vertex} guarantees $\Lambda_{i} \neq \emptyset$, and this process terminates with a subset $\Lambda_{m}$. Note that the subgraph induced by $\Lambda_{i}$ is a star with center $w_{i}$, so we call $w_{i}$ the {\em center of $\Lambda_{i}$}. 
\begin{claim}%
If $i \neq j$, then $\Lambda_{i}$ and $\Lambda_{j}$ have distinct centers. 
\end{claim}
{\noindent\bf Proof of the Claim.}
Suppose to the contrary that there exists $j < i$ such that $\Lambda_{i}$ and $\Lambda_{j}$ have the same center $w$. Let $wv$ be an edge in $\Lambda_{i}$. From the construction, the vertex $w$ is the center of $\Lambda_{j}$, thus $\deg_{j+1}(w) \leq k$ and $\deg_{i}(w) \leq \deg_{j+1}(w) \leq k$. The fact $wv \notin \Lambda_{j}$ implies that $\deg_{j}(v) > k$. Since $wv \in \Lambda_{i}$, it follows that $\deg_{i}(v) \leq k$, and then there exists $j < t < i$ such that $v$ is the center of $\Lambda_{t}$. But $wv \notin \Lambda_{t}$, which leads to a contradiction that $\deg_{j+1}(w) \leq k < \deg_{t}(w)$. This completes the proof of the claim. \qed

We color the edges from $\Lambda_{m}$ to $\Lambda_{1}$, we remind the readers this is the reverse order of the ordinary sequence. 

In the following, we want to give an algorithm to obtain a strong edge coloring of $G$. First, we can color the edges in $\Lambda_{m}$ with distinct colors. Now, we consider the edge $w_{i}v_{i}$, where $w_{i}v_{i} \in \Lambda_{i}$. We want to assign a color to $w_{i}v_{i}$ such that the resulting coloring is still a partial strong edge coloring. In order to guarantee the resulting coloring is a partial strong edge coloring in each step, any two edges which are incident with $N_{G}(w_{i}) \cup N_{G}(v_{i})$ must receive distinct colors.

Now, we compute the number of colored edges in $G_{i}$ which are incident with $N_{G}(w_{i}) \cup N_{G}(v_{i})$. Let $X_{i} = N_{G}(w_{i}) \setminus N_{G_{i}}(w_{i})$, let $x_{i} = |X_{i}|$, let $y_{i} = |\{\,v \mid w_{i}v \in E(G_{i}) \mbox{ and } \deg_{G_{i}}(v) \leq k\,\}|$, and let $z_{i} = |\{\,v \mid w_{i}v \in E(G_{i}) \mbox{ and } \deg_{G_{i}}(v) \geq k\,\}|$. Suppose that $x_{i} > 0$ and let $w$ be an arbitrary vertex in $X_{i}$. By the claim, the edge $ww_{i}$ is in some $\Lambda_{t}$ with center $w$, thus $\deg_{G}(w_{i}) \leq k$ and $w$ is incident with at most $k$ colored edges. Therefore, if $x_{i} \neq 0$, then the number of colored edges which are incident with $N_{G}(w_{i}) \setminus \{v_{i}\}$ is at most
\begin{align*}%
     & (x_{i} + y_{i} - 1) \cdot k + z_{i} \cdot \Delta\\
=    & (x_{i} + y_{i} - 1 + z_{i}) \cdot k + z_{i} \cdot (\Delta - k)\\
\leq & (k - 1)k + (k-2)(\Delta - k), \quad \text{note that $x_{i} + y_{i} + z_{i} = \deg_{G}(w_{i}) \leq k$ and $z_{i} \leq k-2$}\\
=   & (k-2)\Delta + k;
\end{align*}
if $x_{i} = 0$, then the number of colored edges which are incident with $N_{G}(w_{i}) \setminus \{v_{i}\}$ is at most
\begin{align*}%
     & (y_{i} - 1) \cdot k + z_{i} \cdot \Delta\\
=    & (y_{i} - 1 + z_{i}) \cdot k + z_{i} \cdot (\Delta - k)\\
\leq & (\Delta - 1) \cdot k + k \cdot (\Delta - k), \quad \text{note that $z_{i} \leq k$}\\
=    & 2k\Delta - k^{2} -k.
\end{align*}

Let $p_{i} = |N_{G}(v_{i}) \setminus N_{G_{i}}(v_{i})|$, and let $q_{i} = |N_{G_{i}}(v_{i})|$. If $\deg(v_{i}) > k$, then $p_{i} > 0$ and there exists some $s$ with $s < i$ such that $v_{i}$ is the center of $\Lambda_{s}$. It follows that $\deg_{s+1}(v_{i}) \leq k$ and for every edge $uv_{i}$ in $\Lambda_{s}$, the vertex $u$ is incident with at most $k-1$ colored edges. Therefore, if $\deg(v_{i}) > k$, then the number of colored edges which are incident with $N_{G}(v_{i}) \setminus \{w_{i}\}$ is at most
\begin{align*}%
     &(k-1) \cdot (\deg(v_{i}) - \deg_{s+1}(v_{i})) + (\deg_{s+1}(v_{i}) - 1) \cdot \Delta\\
=    & (k-1) \cdot \deg(v_{i}) + \deg_{s+1}(v_{i}) \cdot (\Delta - k + 1) - \Delta\\
\leq & (k-1) \Delta + k \cdot (\Delta - k +1) - \Delta, \quad \text{note that $\deg_{s+1}(v_{i}) \leq k$}\\
=    & 2(k-1)\Delta + k(1-k);
\end{align*}
if $\deg(v_{i}) \leq k$, then the number of colored edges which are incident with $N_{G}(v_{i}) \setminus \{w_{i}\}$ is at most $(k-1) \cdot \Delta$.

Hence, the number of colored edges incident with $N_{G}(w_{i}) \cup N_{G}(v_{i})$ is at most 
\begin{align*}%
 &\max\{(k-2)\Delta + k, 2k\Delta - k^{2} - k\} + \max\{2(k-1)\Delta + k(1-k), (k-1)\Delta\}\\
=&(2k\Delta - k^{2} - k) + 2(k-1)\Delta + k(1-k)\\
=&(4k-2) \Delta - 2k^{2}.
\end{align*}

Thus, there are at least one available color for $w_{i}v_{i}$. When all the edges are colored, we obtain a strong edge coloring of $G$.
\end{proof}

\begin{corollary}\label{Coro}%
If $G$ is a $2$-degenerate graph with maximum degree at least two, then $\chiup_{s}'(G) \leq 6\Delta(G) - 7$.
\end{corollary}
\begin{remark}%
The strong chromatic index of a $5$-cycle is five, so it achieves the upper bound in \autoref{Coro}, thus the obtained upper bound is best possible in some sense. 
\end{remark}

Next, we investigate a class of graphs whose all $3^{+}$-vertices induce a forest. 

\begin{lemma}\label{nice}%
If $G$ is a graph with at least one edge such that all the $3^{+}$-vertices induce a forest, then there exists a vertex $w$ such that at least $\max\{1, \deg_{G}(v) - 1\}$ of its neighbors are $2^{-}$-vertices. 
\end{lemma}
\begin{proof}%
Let $A$ be the set of $2^{-}$-vertices. It is obvious that $A \neq \emptyset$, otherwise every vertex has degree at least three, and then $G$ contains cycles. Now, consider the graph $G - A$. If $G - A$ is an empty graph, then every nonisolated vertex satisfies the desired condition. So we may assume that the  $G - A$ has at least one vertex. Since the graph $G - A$ is a forest, it follows that there exists at least one $1^{-}$-vertex in $G - A$, and every such vertex satisfies the desired condition. 
\end{proof}

\begin{theorem}\label{3+Forest}%
If $G$ is a graph such that all of its $3^{+}$-vertices induce a forest, then $\chiup_{s}'(G) \leq 4 \Delta(G) -3$.
\end{theorem}
\begin{proof}%
The proof is analogous to that in \autoref{k-degenerate}, so we omit it. 
\end{proof}

A $2$-connected graph $G$ is {\em minimally $2$-connected} if $G - e$ is not $2$-connected for each edge $e$ in $G$.

\begin{theorem}[\cite{MR2078877}]%
The minimum degree of a minimally $2$-connected graph is two, and the subgraph induced by all the $3^{+}$-vertices is a forest. 
\end{theorem}
\begin{corollary}%
If $G$ is a minimally $2$-connected graph, then $\chiup_{s}'(G) \leq 4\Delta(G) - 3$.
\end{corollary}
\begin{remark}%
Once again the strong chromatic index of a $5$-cycle achieves the upper bound in \autoref{3+Forest}, so the upper bound is best possible in some sense. 
\end{remark}

A graph is {\em chordless} if every cycle is an induced cycle. It is easy to show that a $2$-connected graph is chordless if and only if it is minimally $2$-connected. Every chordless graph is a $2$-degenerate graph. In \cite{MR3056878}, Chang and Narayanan proved that the strong chromatic index of a chordless graph is at most $8\Delta(G) - 6$, and D{\k e}bski \etal \cite{2013arXiv1301.1992D} improved the upper bound to $4\Delta(G) -3$, but I doubt that their proofs are correct since they incorrectly used a lemma (see \cite[Lemma 7]{MR3056878}). Recently, Narayanan (private communication, Mar. 2014) has confirmed the mistake in the proof of the result $8\Delta(G) - 6$. For more detailed discussion on the strong chromatic index of chordless graphs, we refer the reader to \cite{Basavaraju2013}. 

\begin{remark}%
All the proofs used the greedy algorithm, thus all the results are true for the list version of strong edge coloring.
\end{remark}

\vskip 3mm \vspace{0.3cm} \noindent{\bf Acknowledgments.} The author would like to thank the anonymous reviewers for their valuable comments and assistance on earlier drafts. The author would also like to especially thank N. Narayanan for the discussion on Lemma 7 and Theorem 2 in \cite{MR3056878}. In addition, the author was supported by NSFC (11101125).

\end{document}

%% file: format.tex
\usepackage{mathrsfs, amsmath, amsthm}
\usepackage{graphicx, color}
\usepackage{caption, subcaption}
\usepackage{array}
\usepackage{cases}
\makeatletter
\def\th@plain{%
  \upshape 
}
\makeatother

\makeatletter
\renewenvironment{proof}[1][\proofname]{\par
  \pushQED{\qed}%
  \normalfont \topsep6\p@\@plus6\p@\relax
  \trivlist
  \item[\hskip\labelsep
        \bfseries
    #1\@addpunct{.}]\ignorespaces
}{%
  \popQED\endtrivlist\@endpefalse
}
\makeatother

\newtheorem{theorem}{Theorem}
\numberwithin{theorem}{section}
\newtheorem{lemma}{Lemma}
\newtheorem{corollary}{Corollary}

\newtheorem{conjecture}{Conjecture}
\newtheorem*{conjecture*}{Conjecture}

\newtheorem{claim}{Claim}

\theoremstyle{definition}

\newtheorem{remark}{Remark}

\usepackage[left=25mm,top=25mm,bottom=25mm,right=25mm]{geometry}
\usepackage[T1]{fontenc}
\usepackage[varg]{txfonts}

\usepackage{enumitem}
\usepackage[square, numbers, sort&compress]{natbib}

\ifx\pdfoutput\undefined
 \usepackage[dvipdfm,%
  bookmarks=true,%
  bookmarksnumbered=true, 
  bookmarksopen=true, 
  plainpages=false,%
  pdfpagelabels,%
  colorlinks=true, 
  linkcolor=blue, 
  citecolor=blue,%
  hyperindex=true,
  urlcolor=black]{hyperref}
\else
 \usepackage[pdftex,%
  bookmarks=true,%
  bookmarksnumbered=true, 
  bookmarksopen=true, 
  plainpages=false,%
  pdfpagelabels,%
  colorlinks=true, 
  linkcolor=blue, 
  citecolor=blue,%
  anchorcolor=green,
  urlcolor= blue,
  breaklinks=true,
  hyperindex=true]{hyperref}
\fi

\newcounter{Hcase}
\newcounter{Hclaim}

\newcommand{\etal}{et~al.\ }

\def\int(#1){\mathrm{int}(#1)}
\def\ext(#1){\mathrm{ext}(#1)}
\def\Int(#1){\mathrm{Int}(#1)}
\def\Ext(#1){\mathrm{Ext}(#1)}
\def\ad(#1){\mathrm{ad}(#1)}
\def\mad(#1){\mathrm{mad}(#1)}
\def\la(#1){\mathrm{la}(#1)}

